\documentclass[12pt]{amsart}
\usepackage[utf8]{inputenc}
\usepackage{fullpage}
\usepackage{graphicx}
\usepackage{float}
\usepackage{parskip}
\usepackage{amssymb,amscd,amsbsy,amsmath, amsthm, amsfonts, amsrefs}
\usepackage[OT2,OT1]{fontenc}
\usepackage{color}   
\usepackage[colorlinks=true,linkcolor=blue,citecolor=blue,urlcolor=blue]{hyperref}
\usepackage{algorithmicx} 
\usepackage{algpseudocode} 
\usepackage{dcolumn}
\usepackage{comment}
\includecomment{A}

\newtheorem{theorem}{Theorem}

\newtheorem{lemma}[theorem]{Lemma}

\theoremstyle{definition}
\newtheorem{definition}[theorem]{Definition}

\theoremstyle{remark}

\newfont{\cmbsy}{cmbsy10}
\newfont{\cmmib}{cmmib10}
\newcommand{\Orden}{\mathop{\hbox{\cmbsy O}}\nolimits}

\def\R{\mathbf{R}}
\def\C{\mathbf{C}}
 
\renewcommand{\Re}{\operatorname{Re}}
\renewcommand{\Im}{\operatorname{Im}}

\begin{document}

\title{Computation of the Secondary Zeta function}
\author{J. Arias de Reyna}
\address{Univ.~de Sevilla \\
Facultad de Matem\'aticas \\
c/Tarfia, sn
 \\
41012-Sevilla \\
Spain} 
\email{arias@us.es}

\date{\today}

\begin{abstract}
The secondary zeta function $Z(s)=\sum_{n=1}^\infty\alpha_n^{-s}$, where $\rho_n=\frac12+i\alpha_n$ are the zeros of zeta with $\Im(\rho)>0$, extends to a meromorphic function on the hole complex plane. If we assume the Riemann hypothesis the numbers $\alpha_n=\gamma_n$, but we do not assume the RH.
We give an algorithm to compute this analytic prolongation of the Dirichlet series $Z(s)=\sum_{n=1}^\infty \alpha_n^{-s}$, for all values of $s$ and to a given precision. 
\end{abstract}

\maketitle
\section{Introduction}

The secondary zeta function is defined by 
\begin{equation}\label{def}
Z(s)=\sum_{n=1}^\infty \frac{1}{\alpha_n^s},\qquad \Re s>1
\end{equation}
where $\rho_n= \tfrac12 +i\alpha_n$ runs through the zeros $\rho$ of $\zeta(s)$ with $\Im\rho>0$. 

The function $Z(s)$ extends to a meromorphic function on $\C$, whose poles are double  at $s=1$, and simple at $s=-(2n-1)$  for $n=1$, $2$, \dots. Its properties have been studied by several authors:   Mellin \cite{M}, Cramér \cite{C}, Gui\-nand \cite{G}
Delsarte \cite{D}, Chakravarty \cite{C1}, \cite{C2}, \cite{C3}, \cite{C4}, Ivi\'c
\cite{I}. A detailed summary is found in Voros \cite{V1}, \cite{V2}.
 
Our objective is to  numerically compute the values of the function $Z(s)$ for all values of $s$. The computation of $Z(s)$ is difficult, as witnessed by Voros \cite{V2}*{p.~123} who says: \emph{Apart from very few tables for special values (where special formulae apply), we had not seen such functions tabulated or plotted before, and we wanted to view them.} In Section \ref{x-ray}, we present such plots. 

The difficulty with the computation arises from the fact that $Z(s)$ is represented as the sum of several terms among which there is a great deal of cancellation. Therefore, as explained in  Section \ref{Implem}, the computation of $Z(s)$  requires  high precision.  We have implemented our program, written in Python, by means of the open source library for  multiprecision  floating-point arithmetic, \texttt{mpmath} (see \cite{FJ}). 

The program  is now part of the standard version of \texttt{mpmath} (version 1.1.0). There the code may be  downloaded free of charge.  In \texttt{mpmath}  the function is called \texttt{secondzeta}. 

To compute $Z(s)$, we follow Delsarte in his proof of the prolongation of the Dirichlet series which yields $Z(s)$. In this way we express $Z(s)$ as a combination of four terms
\begin{displaymath}
Z(s)=A(s)-P(s)+E(s)-S(s),
\end{displaymath}
where $A(s)$ depends on the zeros of zeta, $P(s)$ on the primes, $E(s)$ is an entire function related to $\Gamma(s)$, and $S(s)$ the singular term which contains the poles of $Z(s)$ and  an asymptotic series with coefficients that depends on Euler and Bernoulli numbers. All these terms depend on a parameter $a>0$. 

Delsarte, \cite{D} assuming the Riemann Hypothesis, proves a \emph{modular} relation between a sum related to the zeros of zeta and another related to primes. Chakravarty \cite{C1} proves a similar relation without assuming the Riemann Hypothesis. In Theorem \ref{modularT} it is proved that the form given by Delsarte is equivalent to the one given by Chakaravarty, and hence the result of Delsarte does not depends on Riemann Hypothesis.

Throughout the paper we  arrange the zeros $\rho$ of $\zeta(s)$ with $\Im \rho>0$  in a sequence $(\rho_n)$ such that $\Im\rho_{n+1}\ge\Im\rho_n$. Furthermore, if multiple zeros of $\zeta(s)$ exist, the corresponding terms are repeated in the sequence according to their order of multiplicity. We then set $\rho_n= \tfrac12 +i\alpha_n$. Thus if the Riemann Hypothesis is true, then $\alpha_n$ are real numbers (in this case usually denoted by $\gamma_n$), while some of the $\alpha_n$ will be complex if the Riemann Hypothesis is not true. Our $\alpha_n$ are the zeros of $\Xi(t)$ with $\Re\alpha>0$.

\section{The Modular Equation.}

We shall need the following technical but interesting Lemma.

\begin{lemma}\label{lemmatransform}
We have
\begin{equation}\label{transform}
\int_{-\infty}^{+\infty}\Bigl\{\log|t|-\Re\frac{\Gamma'(\frac12+it)}{\Gamma(\frac12+it)}\Bigr\}e^{-2\pi i x t}\,dt=\frac{\pi}{2\sinh\pi|x|}-\frac{1}{2|x|}
\end{equation}
where the function in parentheses belongs to $\mathcal{L}^1(\R)$.
\end{lemma}

\begin{proof}
By the Stirling expansion \cite{L}*{eq. (10), p.~33}, we have
\begin{equation}\label{eq3}
\Re\frac{\Gamma'(\frac12+it)}{\Gamma(\frac12+it)}\sim \log |t|+\Orden(t^{-2}),\qquad t\to\pm\infty.
\end{equation}
From \eqref{eq3} it is easily shown that $\log |t|-\Re\frac{\Gamma'(\frac12+it)}{\Gamma(\frac12+it)}\in\mathcal{L}^1(\R)$.

By diferentiating Binet's formula (see \cite{W}*{[p.~249}) we obtain
\begin{equation}
\frac{\Gamma'(z)}{\Gamma(z)}=\log z-\frac{1}{2z}-\int_0^{+\infty} \Bigl(\frac12-\frac1t+\frac{1}{e^t-1}\Bigr)e^{-tz}\,dt,\qquad \Re z>0,
\end{equation}
and hence 
\begin{equation}
\log\Bigl(s-\frac12\Bigr)-\frac{\Gamma'(s)}{\Gamma(s)}=\log\Bigl(s-\frac12\Bigr)-
\log s+\frac{1}{2s}+\int_0^{+\infty} \Bigl(\frac12-\frac1t+\frac{1}{e^t-1}\Bigr)e^{-ts}\,dt
\end{equation}
which  may be writen as 
\begin{equation}
\log\Bigl(s-\frac12\Bigr)-\frac{\Gamma'(s)}{\Gamma(s)}=\int_0^{+\infty} \Bigl(1-\frac1t e^{\frac{t}{2}}+\frac{1}{e^t-1}\Bigr)e^{-ts}\,dt.
\end{equation}
Therefore
\begin{equation}
\log s-\frac{\Gamma'(\frac12+s)}{\Gamma(\frac12+s)}=\int_0^{+\infty} \Bigl(e^{-\frac{t}{2}}-\frac1t +\frac{e^{-\frac{t}{2}}}{e^t-1}\Bigr)e^{-ts}\,dt=
\int_0^{+\infty} \Bigl(\frac{1}{2\sinh\frac{t}{2}}-\frac1t\Bigr)e^{-ts}\,dt.
\end{equation}
The function in brackets in the integrand is in $\mathcal{L}^2(\R)$ (taking it equal to $0$ for $t<0$).  Hence its Fourier transform is obtained when we put  $s=2\pi i x$
\begin{equation}
\mathcal{F}\Bigl[\Bigl(\frac{1}{2\sinh\frac{t}{2}}-\frac1t\Bigr)
\chi_{[0,+\infty)}(t)\Bigr]=\log(2\pi i x)-\frac{\Gamma'(\frac12+2\pi i x)}{\Gamma(\frac12+2\pi i x)}.
\end{equation}
Taking only the real parts, the cosine transform becomes
\begin{equation}
\mathcal{C}\Bigl[\Bigl(\frac{1}{2\sinh\frac{t}{2}}-\frac1t\Bigr)
\chi_{[0,+\infty)}(t)\Bigr]=\log(2\pi |x|)-\Re\frac{\Gamma'(\frac12+2\pi i x)}{\Gamma(\frac12+2\pi i x)}.
\end{equation}
The function on the right is both even and in $\mathcal{L}^1(\R)$ and hence
\begin{equation}
\int_{-\infty}^{+\infty}\Bigl(
\log(2\pi |x|)-\Re\frac{\Gamma'(\frac12+2\pi i x)}{\Gamma(\frac12+2\pi i x)}\Bigr)e^{-2\pi i x t}\,dx=
\frac{1}{4\sinh\frac{|t|}{2}}-\frac1{2|t|}.
\end{equation}
If $x$ is replaced with $x/2\pi$, and $t$ with $2\pi t$, then equation \eqref{transform}
is obtained.
\end{proof}

\begin{theorem}\label{modularT}
For $x>0$ 
\begin{equation}\label{modular}
\sum_{n=1}^\infty e^{-\alpha_n^2 x}=
-\frac{1}{2\sqrt{\pi x}}
\sum_{n=2}^\infty \frac{\Lambda(n)}{\sqrt{n}}e^{-\frac{1}{4x}\log^2n}+e^{\frac{x}{4}}
-\Phi(x),
\end{equation}
where
\begin{multline}\label{equality}
\Phi(x):=\frac{1}{2\pi}\int_0^{+\infty}e^{- xt^2}\Psi(t)\,dt=\\
=
\frac{C_0+\log 16\pi^2x}{8\sqrt{\pi x}}-\frac{1}{4\sqrt{\pi x}}\int_0^{+\infty}e^{-\frac{u^2}{16x}}\Bigl(\frac{1}{u}-\frac{e^{\frac34 u}}{e^u-1}\Bigr)\,du
\end{multline}
and 
\begin{equation}\label{defPsi}
\Psi(t):=\log2\pi+\frac{\pi}{2\cosh\pi t}-\Re\Bigl\{\frac{\Gamma'(\frac12+it)}
{\Gamma(\frac12+it)}\Bigr\}.
\end{equation}
and $C_0$ is Euler constant.
\end{theorem}

\begin{proof}
Delsarte, \cite{D}*{eq.~(19) and (25), p.~424, 426}, by assuming the  Riemann Hypothesis, proves \eqref{modular} with $\Phi(x)$  given by $\frac{1}{2\pi}\int_0^{+\infty}e^{-xt^2}\Psi(t)\,dt$.  

On the other hand Chakravarty, \cite{C1}*{eq.~(3.2), p.~286}, unconditionally  proves \eqref{modular} but with $\Phi(x)$ given by the second expression in \eqref{equality}

If we prove equation \eqref{equality}, then the Theorem  follows. (Surely  equality \eqref{equality} does not depend on the Riemann Hypothesis!). 

We may write 
\begin{equation}
\Psi(t)=-\log\frac{|t|}{2\pi}+\frac{\pi}{2\cosh\pi t}+\Bigr(\log |t|-\Re\Bigl\{\frac{\Gamma'(\frac12+it)}
{\Gamma(\frac12+it)}\Bigr\}\Bigr).
\end{equation}
It is easy to see that 
\begin{equation}
\frac{1}{2\pi}\int_0^{+\infty}\log\frac{t}{2\pi}e^{-xt^2}\,dt=-\frac{C_0+\log(16\pi^2 x)}{8\sqrt{\pi x}}.
\end{equation}
The Fourier transforms are also well-known:
\begin{equation}
\int_{-\infty}^{+\infty}\frac{\pi}{2\cosh\pi t}e^{-2\pi i y t}\,dt=\frac{\pi}{2\cosh\pi y},\quad 
\int_{-\infty}^{+\infty}e^{-x t^2}e^{-2\pi i y t}\,dt=\sqrt{\frac{\pi}{x}}e^{-\frac
{\pi^2 y^2}{x}}.
\end{equation}
Hence by means of Lemma \ref{lemmatransform} and Parseval we have
\begin{multline}
\frac{1}{2\pi}\int_{-\infty}^{+\infty}\Bigl\{\frac{\pi}{2\cosh\pi t}+\Bigr(\log |t|-\Re\frac{\Gamma'(\frac12+it)}
{\Gamma(\frac12+it)}\Bigr)\Bigr\}e^{-xt^2}\,dt=\\=\frac{1}{2\pi}
\int_{-\infty}^{+\infty}\Bigl\{
\frac{\pi}{2\cosh\pi y}+\frac{\pi}{2\sinh\pi|y|}-\frac{1}{2|y|}
\Bigr\}\sqrt{\frac{\pi}{x}}e^{-\frac
{\pi^2 y^2}{x}}\,dy=
\end{multline} 
\begin{equation}
=\frac{1}{2\sqrt{\pi x}}\int_{-\infty}^{+\infty}\Bigl\{
\frac{2\pi e^{3\pi |y|}}{e^{4\pi |y|}-1}-\frac{1}{2|y|}
\Bigr\}e^{-\frac
{\pi^2 y^2}{x}}\,dy=
\frac{1}{2\sqrt{\pi x}}\int_{-\infty}^{+\infty}\Bigl\{
\frac{2\pi e^{\frac34|u|}}{e^{|u|}-1}-\frac{2\pi}{|u|}
\Bigr\}e^{-\frac
{u^2}{16x}}\,\frac{du}{4\pi}.
\end{equation}
Since the integrated functions are even, this is equivalent to equation \eqref{equality}.
\end{proof}

\section{Main formula.}

\begin{definition}
For $\Re z>0$ we define
\begin{equation}
f(z):=\sum_{n=1}^\infty e^{-\alpha_n^2z}.
\end{equation}
The function $f$ is holomorphic in $\Re z>0$.
\end{definition}

\begin{proof}
We must show that the series converges uniformly in compact sets of $\Re z>0$.  Let $K$ be such  a compact. There exists $\delta>0$ and $M<\infty$ such that $\Re z\ge \delta>0$ and $|z|\le M$ for $z\in K$. 

By setting $\alpha_n=\gamma_n+i\beta_n$,  its real and imaginary parts,  and $z=x+iy\in K$, we have
\begin{displaymath}
|e^{-\alpha_n^2z}|= e^{-(\gamma_n^2-\beta_n^2)x+2\gamma_n\beta_n y }\le e^{-\gamma_n^2 \delta+M\beta_n^2+2M\gamma_n|\beta_n|}.
\end{displaymath}
Since $|\beta_n|\le 1/2$,  for a sufficiently large $n$  we obtain
\begin{displaymath}
|e^{-\alpha_n^2z}|\le C e^{-\gamma_n^2 \delta+M\gamma_n}\le C e^{-\frac{\delta}{2}\gamma_n^2}.
\end{displaymath}
Since $\gamma_n\sim n\log n$,  the numerical series $\sum e^{-\frac{\delta}{2}\gamma_n^2}$ converges. 
\end{proof}

\begin{theorem}
For $\Re s>1$, 
\begin{equation}\label{modexp}
\Gamma(s/2)Z(s)=\int_0^{+\infty}f(x)x^{\frac{s}{2}}\frac{dx}{x}.
\end{equation}
\end{theorem}

\begin{proof}
For each $n$,  $\Re\alpha_n=\gamma_n\ge\gamma_1=\delta> \tfrac12 $, and hence for $x>0$ real
$|e^{-\alpha^2 x}|=e^{-(\Re\alpha)^2 x+(\Im\alpha)^2x}\le e^{-(\delta^2-\frac14)x}=e^{-bx}$ with $b>0$. It follows that, for each $n$,
\begin{equation}
\int_0^\infty e^{-\alpha_n^2 x}x^{\frac{s}{2}}\frac{dx}{x}=
\frac{\Gamma(s/2)}{\alpha_n^s},
\end{equation}
and therefore
\begin{equation}
\int_0^\infty f(x)x^{\frac{s}{2}}\frac{dx}{x}=\sum_{n=1}^\infty
\int_0^\infty e^{-\alpha_n^2 x}x^{\frac{s}{2}}\frac{dx}{x}=\Gamma(s/2)Z(s).
\end{equation}
The sum and the integral can be interchanged since 
$\sum \{\gamma_n^2-\frac14\}^{-\sigma/2}<+\infty$ for $\sigma>1$. 
\end{proof}

We shall take a positive real number $a$, and separate $Z(s)$ into two parts
\begin{gather}
Z(s)=A(s)+B(s),\\ A(s)=\frac{1}{\Gamma(\frac{s}{2})}\int_a^{+\infty}f(x)
x^{\frac{s}{2}}\frac{dx}{x},\qquad 
B(s)=\frac{1}{\Gamma(\frac{s}{2})}\int_0^af(x)
x^{\frac{s}{2}}\frac{dx}{x}.
\end{gather}

\begin{lemma}\label{incomplete}
The following bounds for the incomplete gamma function holds:
\begin{equation}\label{Gabcke}
\Gamma(\sigma, a)\le \sigma e^{-a}a^{\sigma-1},\qquad a>\sigma\ge1.
\end{equation}
\begin{equation}\label{second}
\Gamma(\sigma, a)\le e^{-a}a^{\sigma-1},\qquad a>0,\qquad \sigma\le1.
\end{equation}
\end{lemma}

\begin{proof}
The first inequality \eqref{Gabcke} may be found in \cite{Ga}*{Satz 3, p.~145}.
The second is well-known (see Olver \cite{O}*{(1.05) p.~67}) and easy to prove. 
Observe that  for $\sigma\le1$, 
\begin{equation}
a^{1-\sigma}\int_a^\infty t^\sigma e^{-t}\frac{dt}{t}=
\int_a^\infty \Bigl(\frac{t}{a}\Bigr)^{\sigma-1} e^{-t}\,dt<
\int_a^\infty  e^{-t}\,dt=e^{-a}.\qedhere
\end{equation}
\end{proof}

\begin{theorem}
$A(s)$ extends to an entire function that is given by the series
\begin{equation}\label{A}
A(s)=\sum_{n=1}^\infty \frac{\Gamma(\frac{s}{2},a\alpha_n^2)}{\Gamma(\frac{s}{2})}
\frac{1}{\alpha_n^s}.
\end{equation}
\end{theorem}
\begin{proof}
Consider the integral
\begin{equation}
\int_a^\infty \sum_{n=1}^\infty e^{-\alpha_n^2x}x^{\frac{s}{2}-1}\,dx
\end{equation}
we have
\begin{equation}\label{bound}
\sum_{n=1}^\infty\int_a^\infty |e^{-\alpha_n^2x}x^{\frac{s}{2}-1}|\,dx\le 
\sum_{n=1}^\infty\int_a^\infty e^{-(\gamma_n^2-\frac14)x}x^{\frac{\sigma}{2}-1}\,dx
=\sum_{n=1}^\infty \frac{\Gamma(\frac{\sigma}{2},a(\gamma_n^2-\frac14) )}{(\gamma_n^2-\frac14)^{\frac{\sigma}{2}}}.
\end{equation}

The series in \eqref{bound} is convergent by Lemma \ref{incomplete}. The 
integral and the sum may then be exchanged in order to obtain \eqref{A}. It is clear that the bounds are uniform on compact sets.  Therefore $A(s)$ extends to an integral function.
\end{proof}

We shall need an estimation of the error commited if  the sum in \eqref{A} is substituted with
a partial sum.
We have
\begin{equation}
R_N:=\sum_{N+1}^\infty \Bigl|\frac{\Gamma(\frac{s}{2},a\alpha_n^2)}{\Gamma(\frac{s}{2})}\frac{1}{\alpha_n^s}\Bigr|\le 
\frac{1}{|\Gamma(\frac{s}{2})|}\sum_{N+1}^\infty\frac{\Gamma(\frac{\sigma}{2},a(\gamma_n^2-\frac14))}
{(\gamma_n^2-\frac14)^{\frac{\sigma}{2}}}.
\end{equation}
We now assume  that $N$ is sufficiently large to guarantee $a(\gamma_N^2-\frac14)>\frac{\sigma}{2}$ and define
$\mu:=\max(\frac{\sigma}{2},1)$. Lemma \ref{incomplete} may therefore be applied in order to obtain 
\begin{equation}
R_N\le \frac{\mu}{|\Gamma(\frac{s}{2})|}\sum_{N+1}^\infty e^{-a(\gamma_n^2-\frac14)}
\frac{(a(\gamma_n^2-\frac14))^{\frac{\sigma}{2}-1}}{(\gamma_n^2-\frac14)^{\frac{\sigma}{2}}}=
\frac{\mu a^{\frac{\sigma}{2}-1}}{|\Gamma(\frac{s}{2})|}\sum_{N+1}^\infty\frac{e^{-a(\gamma_n^2-\frac14)}}
{\gamma_n^2-\frac14}.
\end{equation}
Hence 
\begin{equation}
R_N\le \frac{2\mu a^{\frac{\sigma}{2}-1}e^\frac{a}{4}}{|\Gamma(\frac{s}{2})|}\sum_{N+1}^\infty\frac{e^{-a\gamma_n^2}}
{\gamma_n^2}= \frac{2\mu a^{\frac{\sigma}{2}-1}e^\frac{a}{4}}{|\Gamma(\frac{s}{2})|}
\int_{\gamma_N}^\infty\frac{e^{-ax^2}}
{x^2}\,dN(x).
\end{equation}
The function $N(x)=\frac{1}{2\pi}\log\frac{x}{2\pi}-\frac{x}{2\pi}+R(x)$ with $R(x)=\Orden(\log x)$. Therefore $N(x)$ increases as if it had a density of $\frac{1}{2\pi}\log\frac{x}{2\pi}$, and hence, not rigorously, 
a bound may be obtained, 
\begin{equation}\label{error-main}
R_N\lessapprox \frac{2\mu a^{\frac{\sigma}{2}-1}e^\frac{a}{4}}{2\pi|\Gamma(\frac{s}{2})|}\log\frac{\gamma_N}{2\pi}
\int_{\gamma_N}^\infty\frac{e^{-ax^2}}
{x^2}\,dx=\frac{\mu a^{\frac{\sigma-1}{2}}e^\frac{a}{4}}{2\pi}\frac{\Gamma(-\frac12,a\gamma_N^2)}{|\Gamma(\frac{s}{2})|}\log\frac{\gamma_N}{2\pi}.
\end{equation}
and we have 
\[\frac{\mu a^{\frac{\sigma-1}{2}}e^\frac{a}{4}}{2\pi}\frac{\Gamma(-\frac12,a\gamma_N^2)  }{|\Gamma(\frac{s}{2})|}\log\frac{\gamma_N}{2\pi}\le \frac{\mu a^{\frac{\sigma-1}{2}}e^\frac{a}{4}}{2\pi}\frac{ (a\gamma_N^2)^{-3/2}e^{-a\gamma_N^2}}{|\Gamma(\frac{s}{2})|}\log\frac{\gamma_N}{2\pi}=
\frac{\mu a^{\frac{\sigma}{2}-2}e^\frac{a}{4}}{2\pi\gamma_N^3}\frac{ e^{-a\gamma_N^2}}{|\Gamma(\frac{s}{2})|}\log\frac{\gamma_N}{2\pi}\]
On the other hand 
\[\int_{\gamma_N}^\infty \frac{e^{-ax^2}}{x^2}\,dR(x)=-\frac{e^{-a\gamma_N^2}}{\gamma_N^2}R(\gamma_N)+\int_{\gamma_N}^\infty R(x)e^{-ax^2}\Bigl(\frac{2}{x^3}+\frac{2a}{x}\Bigr)\,dx\]
\[\Bigl|\int_{\gamma_N}^\infty \frac{e^{-ax^2}}{x^2}\,dR(x)\Bigr|\ll
\frac{e^{-a\gamma_N^2}}{\gamma_N^2}\log\gamma_N+\log\gamma_N\int_{\gamma_N}^\infty e^{-ax^2}\Bigl(\frac{2}{x^3}+\frac{2a}{x}\Bigr)\,dx\]
\[\ll \frac{e^{-a\gamma_N^2}}{\gamma_N^2}\log\gamma_N-\log\gamma_N\int_{\gamma_N}^\infty d\frac{e^{-ax^2}}{x^2}=2\frac{e^{-a\gamma_N^2}}{\gamma_N^2}\log\gamma_N.\]
This is also small for $N\to\infty$, and possibly this bound is too large, because of cancellation in the values of $R(x)$. 

In the integral defining $B(s)$, we substitute the modular equation  to obtain three terms 
\begin{gather}
B(s)=-P(s)+E(s)-S(s),\\
P(s)=\frac{1}{\Gamma(\frac{s}{2})}\int_0^a\frac{1}{2\sqrt{\pi x}}
\sum_{n=2}^\infty \frac{\Lambda(n)}{\sqrt{n}}e^{-\frac{1}{4x}\log^2n}x^{\frac{s}{2}}
\frac{dx}{x},\label{defP} \\E(s)=\frac{1}{\Gamma(\frac{s}{2})}\int_0^a e^{\frac{x}{4}}
x^{\frac{s}{2}}\frac{dx}{x},\qquad
S(s)=\frac{1}{\Gamma(\frac{s}{2})}\int_0^a \Phi(x)
x^{\frac{s}{2}}\frac{dx}{x}.
\end{gather}
We shall say that $P(s)$ is the prime term, $E(s)$ the exponential term and $S(s)$ the singular term.

\begin{theorem}\label{prime}
The prime term $P(s)$ extends to an entire function given by the series
\begin{equation}\label{eqP}
P(s)=\frac{1}{2\sqrt{\pi}}\sum_{n=2}^\infty \frac{\Lambda(n)}{\sqrt{n}}
\frac{\Gamma(\frac{1-s}{2},\frac{\log^2n}{4a})}{\Gamma(\frac{s}{2})}
\Bigl(\frac{2}{\log n}\Bigr)^{1-s}.
\end{equation}
\end{theorem}

\begin{proof}
For a real $s$,  the series in the integrand in the definition \eqref{defP} of $P(s)$  has positive terms. In this case  the sum and integral may be interchanged. Furthermore, it is easy to see that if this series is convergent for $s=\sigma$,  the sum and integral  may be interchanged for a $\Re s>\sigma$. 

What is obtained after interchange is the series in \eqref{eqP}.  Hence it only remains to be proved that this series is absolutely convergent,  to justify the interchange of sum and integral.  It is preferable to bound the rest of the sum. 
\begin{multline}
|P_N|\le \frac{1}{2\sqrt{\pi}}\Bigl|\sum_{n=N+1}^\infty \frac{\Lambda(n)}{\sqrt{n}}
\frac{\Gamma(\frac{1-s}{2},\frac{\log^2n}{4a})}{\Gamma(\frac{s}{2})}
\Bigl(\frac{2}{\log n}\Bigr)^{1-s}\Bigr|\le \\ \le
\frac{1}{2\sqrt{\pi}|\Gamma(s/2)|}\sum_{n=N+1}^\infty \frac{\Lambda(n)}{\sqrt{n}}
\Gamma\Bigl(\frac{1-\sigma}{2},\frac{\log^2n}{4a}\Bigr)
\Bigl(\frac{2}{\log n}\Bigr)^{1-\sigma}.
\end{multline}
By first assuming that $\frac{1-\sigma}{2}\ge1$, and by taking a sufficiently large  $N$  we will have
$\frac{\log^2N}{4a}>\frac{1-\sigma}{2}$,  and by \eqref{Gabcke}
\begin{equation}\label{errorP}
|P_N|\le\frac{1-\sigma}{2} \frac{2a^{\frac{1+\sigma}{2}}}{\sqrt{\pi}
|\Gamma(s/2)|}\sum_{n=N+1}^\infty \frac{\Lambda(n)}{\sqrt{n}\log^2n}
e^{-\frac{\log^2n}{4a}}.
\end{equation}
The only difference when  $\frac{1-\sigma}{2}<1$ (by \eqref{second}) is that the factor $\frac{1-\sigma}{2}$ must be eliminated.

Now
\begin{equation}
R_N:=\sum_{n=N+1}^\infty \frac{\Lambda(n)}{\sqrt{n}\log^2n}
e^{-\frac{\log^2n}{4a}}\le \sum_{n=N+1}^\infty \frac{1}{\sqrt{n}\log n}
e^{-\frac{\log^2n}{4a}}.
\end{equation}
The terms of this sum are decreasing and hence
\begin{equation}
R_N\le\int_N^\infty \frac{1}{\sqrt{u}\log u}
e^{-\frac{\log^2u}{4a}}\,du=\int_{\frac{\log N}{2\sqrt{a}}}^\infty
e^{-t^2+\sqrt{a} t}\frac{dt}{t}.
\end{equation}
Multiplying the integrand by $\frac{4at^2}{\log^2N}>1$  yields
\begin{multline}
R_N\le \frac{4a}{\log^2N}
\int_{\frac{\log N}{2\sqrt{a}}}^\infty
t\,e^{-t^2+\sqrt{a} t}\,dt\le \frac{4ae^{\frac{a}{4}}}{\log^2N}
\int_{\frac{\log N}{2\sqrt{a}}}^\infty
te^{-(t-\frac{\sqrt{a}}{2})^2} \,dt\le\\
\le \frac{4ae^{\frac{a}{4}}}{\log^2N}
\int_{\frac{\log N}{2\sqrt{a}}-\frac{\sqrt{a}}{2}}^\infty
(u+ \tfrac12 \sqrt{a})e^{-u^2} \,du.
\end{multline} 
Taking $\log N>2a$, gives $u>\sqrt{a}/2$ and therefore
\begin{equation}
R_N\le \frac{4ae^{\frac{a}{4}}}{\log^2N}
\int_{\frac{\log N}{2\sqrt{a}}-\frac{\sqrt{a}}{2}}^\infty
2u\,e^{-u^2} \,du= \frac{4ae^{\frac{a}{4}}}{\log^2N}e^{-(\frac{\log N}{2\sqrt{a}}
-\frac{\sqrt{a}}{2})^2}.
\end{equation}
For practical purposes  $0<a<1$ is always assumed. Hence $\log N>2a$ will be true if  $N\ge8$. In this case,  $\frac12\frac{\log N}{2\sqrt{a}}>\frac{\sqrt{a}}{2}$. It follows that 
\begin{equation}
R_N\le \frac{4ae^{\frac{a}{4}}}{\log^2N}e^{-\frac{\log^2 N}{16a}}\le 
\frac{6}{\log^2N}e^{-\frac{\log^2 N}{16a}}.
\end{equation}

Hence for  $N\ge8$,
\begin{equation}\label{error46}
P_N\le 
\delta\frac{2a^{\frac{1+\sigma}{2}}}{\sqrt{\pi}
|\Gamma(s/2)|}\frac{6}{\log^2N}e^{-\frac{\log^2 N}{16a}},
\end{equation}
where $\delta=\max(1,\frac{1-\sigma}{2})$. 

Since this bound is uniform on compact sets, the function $P(s)$ is entire. 
\end{proof}

\begin{theorem}
The exponential term extends to  an entire function. It can be computed by the formula
\begin{equation}\label{computE}
E(s)=\frac{1}{\Gamma(\frac{s}{2})} \sum_{n=0}^\infty \frac{1}{4^n n!}\frac{a^{n+\frac{s}{2}}}{n+\frac{s}{2}}
\end{equation}
and for $s=-2n$  we have
\begin{equation}\label{valuesE}
E(-2n)=\frac{(-1)^n}{4^n},\qquad n=0, 1, 2,\dots.
\end{equation}
\end{theorem}

\begin{proof}
For $\Re s>0$,
\begin{equation}
E(s)=\frac{1}{\Gamma(\frac{s}{2})}\int_0^a e^{\frac{x}{4}}x^{\frac{s}{2}}\frac{dx}{x}
=\frac{1}{\Gamma(\frac{s}{2})}\sum_{n=0}^\infty \frac{1}{4^n n!}\int_0^a x^{n+\frac{s}{2}}\frac{dx}{x}=
\frac{1}{\Gamma(\frac{s}{2})} \sum_{n=0}^\infty \frac{1}{4^n n!}\frac{a^{n+\frac{s}{2}}}{n+\frac{s}{2}}.
\end{equation}
From the above expression it is clear that $E(s)$ extends to  an entire function, since the poles $s=-2n$ of the sum cancel out the poles of $\Gamma(\frac{s}{2})$.  The value of $E(-2n)$ is now straightforward to obtain.
\end{proof}

\begin{theorem}
For $x>0$ and $x\to0$, 
\begin{equation}\label{PhiAsympt}
\Phi(x)=\frac{C_0+\log16\pi^2x}{8\sqrt{\pi x}}+\frac{1}{8\sqrt{\pi}}
\sum_{n=1}^N\frac{2^{2n}B_n(3/4)}{n!}\Gamma(n/2)x^{\frac{n-1}{2}}+\Orden(
x^{\frac{N}{2}})
\end{equation}
for all integers $N\ge1$.
\end{theorem}

\begin{proof}
By \eqref{equality} it suffices to show that
\begin{equation}
\frac{1}{4\sqrt{\pi x}}\int_0^{+\infty}e^{-\frac{u^2}{16x}}\Bigl(\frac{e^{\frac34 u}}{e^u-1}-\frac{1}{u}\Bigr)\,du
=\frac{1}{8\sqrt{\pi}}
\sum_{n=1}^N\frac{2^{2n}B_n(3/4)}{n!}\Gamma(n/2)x^{\frac{n-1}{2}}+\Orden(
x^{\frac{N}{2}}).
\end{equation}
From the definition of the Bernoulli polynomials ,
\begin{equation}
\frac{e^{\frac34 u}}{e^u-1}-\frac{1}{u}=\frac{1}{u}\Bigl(
\frac{ue^{\frac34 u}}{e^u-1}-1\Bigr)=\sum_{n=1}^\infty\frac{B_n(3/4)}{n!}u^{n-1}
\end{equation}
Now the result is consequence of the Watson Lemma \cite{L}*{p.~4--5}.
\end{proof}

\begin{theorem}
$S(s)$ extends to a meromorphic function with poles at $s=1$ and $s=-(2n-1)$ for $n=1$, $2$, \dots

Let $N$ be a natural number, and   $\Re s> -N-1$. Then   for $a\to0$ with $a>0$  
\begin{multline}\label{singular}
S(s)=\frac{a^{\frac{s-1}{2}}}{4\sqrt{\pi}\Gamma(s/2)}\Bigl\{
\Bigl(-\frac{2}{(s-1)^2}+\frac{C_0+\log(16\pi^2 a)}{(s-1)}\Bigr)+\\+
\sum_{n=1}^N \frac{B_n(3/4)}{n!}\frac{(4\sqrt{a})^n \Gamma(\frac{n}{2})}{s+n-1}\Bigr\}+\Orden(a^{\frac{N+\sigma}{2}}),
\end{multline}
where the implicit constant depends on $s$ and $N$.
\end{theorem}

\begin{proof}
By definition, $S(s)=\Gamma(\frac{s}{2})^{-1}\int_0 ^a\Phi(x) x^{\frac{s}{2}}
\frac{dx}{x}$. We substitute the value of $\Phi(x)$ given in \eqref{PhiAsympt}. 
Since
\begin{equation}
\frac{1}{\Gamma(\frac{s}{2})}\int_0^a \frac{C_0+\log16\pi^2x}{8\sqrt{\pi x}}
x^{\frac{s}{2}}\frac{dx}{x}=\frac{a^{\frac{s-1}{2}}}{4\sqrt{\pi}\Gamma(\frac{s}{2})}
\Bigl(-\frac{2}{(s-1)^2}+\frac{C+\log(16a\pi^2)}{s-1}\Bigr)
\end{equation}
we have
\begin{multline}
S(s)=\frac{a^{\frac{s-1}{2}}}{4\sqrt{\pi}\Gamma(\frac{s}{2})}
\Bigl(-\frac{2}{(s-1)^2}+\frac{C+\log(16a\pi^2)}{s-1}\Bigr)+\\ +
\frac{1}{\Gamma(\frac{s}{2})}\int_0^a \Bigl(
\frac{1}{8\sqrt{\pi}}
\sum_{n=1}^N\frac{2^{2n}B_n(3/4)}{n!}\Gamma(n/2)x^{\frac{n-1}{2}}+R_N(x)
\Bigr)
x^{\frac{s}{2}}\frac{dx}{x}
\end{multline}
where $R_N(x)$ is continuous on $(0,1)$ and where $R_N(x)$  is $\Orden(
x^{\frac{N}{2}})$ for $x\to 0^+$.

Through  integration 
\begin{multline}
S(s)=\frac{a^{\frac{s-1}{2}}}{4\sqrt{\pi}\Gamma(\frac{s}{2})}
\Bigl(-\frac{2}{(s-1)^2}+\frac{C+\log(16a\pi^2)}{s-1}\Bigr)+\\ +
\frac{1}{\Gamma(\frac{s}{2})} \Bigl(
\frac{1}{8\sqrt{\pi}}
\sum_{n=1}^N\frac{2^{2n}B_n(3/4)}{n!}\Gamma(n/2)\frac{2a^{\frac{s+n-1}{2}}}
{s+n-1}\Bigr)+
\frac{1}{\Gamma(\frac{s}{2})}\int_0^aR_N(x)x^{\frac{s}{2}}\frac{dx}{x}.
\end{multline}

Since the integral of $R_N(x)$ is convergent and defines an analytic function for 
$\sigma>-N$, it follows that $S(s)$ extends to a meromorphic function in that range.

When  $N$ is moved then $S(s)$ is meromorphic on the plane, and it is easy to see that its poles are $s=1$ double and $s=-(2n-1)$ simple for $n=1$, $2$, \dots

Finally, \eqref{singular} follows from the easy bound of the integral of $R_N(s)$  applying that $R_N(x)=\Orden(x^{\frac{N}{2}})$.
\end{proof}

\section{Implementation.}\label{Implem}
We need arbitrary  high precision  in the computations. In fact the four terms in the decomposition $Z(s)=A(s)-P(s)+E(s)-S(s)$ are usually  very large compared with the value of $Z(s)$.  The large amount of cancellation render the   use of high precision necessary. For example when  $Z=Z(0.5+100i)$ is computed with $a=0.015$  the following values of the real parts of the four components are obtained. (The imaginary parts are similar).
\medskip

\begin{tabular}{D..{-1}D..{-1}{c}|}
\multicolumn{2}{c}{$a=0.015$}\\
\noalign{\smallskip}
\hline
\noalign{\smallskip}
$\text{Re}(A)$ & -1217647861137338225487423130072.80874450435043989\\
$\text{Re}(P)$ & 1391575559381350246767527628.7744199492782667711\\
$\text{Re}(E)$ & 18001556143696111321715567598902.9610490404231213\\
$\text{Re}(S)$ & 16782516706999391745981376941201.5941565980711322\\
\noalign{\smallskip}
\hline
\noalign{\smallskip}
$\text{Re}(Z)$ & -0.21627201127671758
\end{tabular}

If  the value of $a$ is changed then completely different values are obtained

\begin{tabular}{D..{-1}D..{-1}{c}}
\multicolumn{2}{c}{$a=0.005$}\\
\noalign{\smallskip}
\hline
\noalign{\smallskip}
$\text{Re}(A)$ & 38410059566483986070633575784051.1823990792080925\\
$\text{Re}(P)$ & -4537279027536395371164.63554410664485168750617707\\
$\text{Re}(E)$ & -71944459339240131526392135204651.586794692259377\\
$\text{Re}(S)$ & -33534399768218866428222164049435.5525794951297153\\
\noalign{\smallskip}
\hline
\noalign{\smallskip}
$\text{Re}(Z)$ & -0.21627201127671736
\end{tabular}

In the first case, with $a=0.015$,  $A$ is computed using 35 zeros of zeta and 29 values of the von Mangoldt function. In the second case,  with $a=0.005$,  the program needs 75 zeros of zeta and 9 values of the von Mangoldt function. 

Hence,  a library that computes special functions to high precision and also computes the zeros of zeta is needed. Mathematica,  may be used, although in order to compute zeros of zeta with high precision it is better to use \texttt{mpmath} \cite{FJ}, which is also  free and open source.

The term $E(s)$ is usually the major term and may be computed easily, and hence this term is employed to determine to what degree of precision  the computation must be carried out, given the approximation required on the final result. 

By adding the terms  of the sum  \eqref{A} together (all the terms greater than the current epsilon of \texttt{mpmath}), the main term is calculated. The error commited is then estimated by applying formula \eqref{error-main}.

To compute  $P(s)$,  the terms of \eqref{eqP} are added together while they are  greater in absolute value than the current epsilon of mpmath. The estimation \eqref{error46} of the error is not applied since it is unrealistic and usually large. It is instead  assumed that the error may be estimated by the first ommited term. The series \eqref{eqP} is adequately convergent if a sufficiently small $a$ is taken. 

Since a sufficiently small $a$ is always taken, the series \eqref{computE} is highly efficient for the calculation of the exponential term $E(s)$, except when $s=-2n$ where the exact value given in \eqref{valuesE} is used directly. The error in the computation of $E(s)$ is always very small and no estimation of it is computed.

The error of the expression \eqref{singular} of the singular term is difficult to obtain rigorously. It is a typical asymptotic expansion  for which  the error is usually approximately of the order of the first term omitted. 

In this case (of the singular term) the odd terms are very different in magnitude from the even terms ($|t_{2n}|<|t_{2n\pm1}|$). Therefore the terms are added together until $|t_{2n-1}|$ is smaller than the mpmath epsilon, or these terms start to increase. The estimation of the error has to take into account  the fact that  these terms are multiplied by a large factor. We estimate the magnitude of this factor, and adequately increase  the precision of mpmath in order to obtain an error of the desired magnitude. The factors must be recomputed to this new precision.   

There is another way to estimate the error. If  $Z(s)$ is computed using two different values of the parameter $a$, then the coincident digits are very probably correct. This procedure may be repeated several times.  In practice, there is a limitation to the values of $a$. A very small value of  $a$  needs many zeros of zeta (\texttt{mpmath} will compute reliably $\gamma_n$ for $1\le n\le 10^{16}$). A large value of $a$ will need many values of the von Mangoldt function $\Lambda(n)$. 

The program is now part of the standard version of  \texttt{mpmath} (version 1.1.0), from which  our code in Python may be obtained. 

The program contains  an option to provide on  estimation of the error, and another to print the values of the four numbers $A$, $P$, $E$ and $S$,  the number of zeros of zeta used on the computation and the number of values of the von Mangoldt function used.

\section{X ray and plots  of \texorpdfstring{$Z(s)$}{Z(s)}.}\label{x-ray}

\begin{figure}[H]
\begin{center}
\includegraphics[width=0.9\hsize]{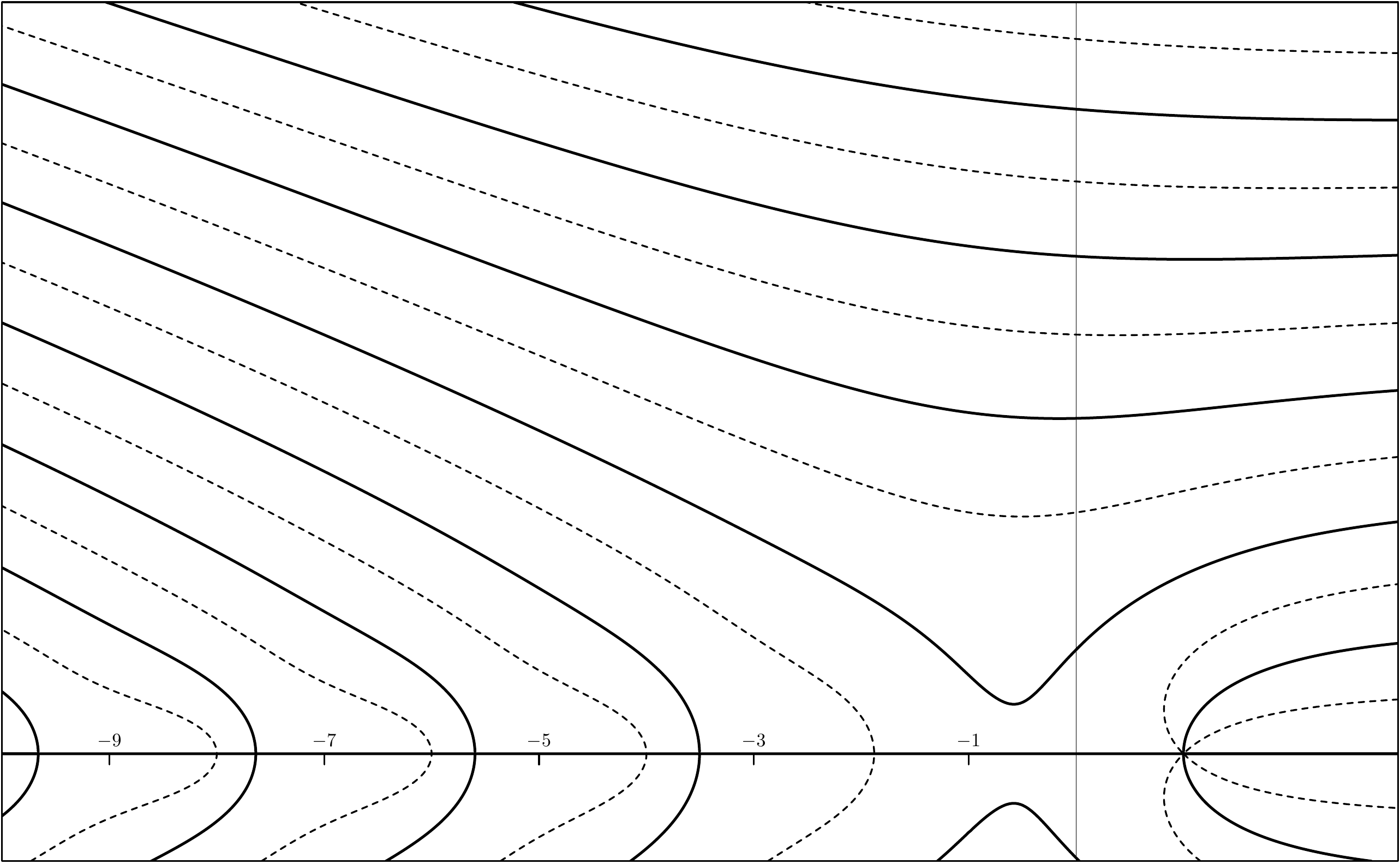}
\caption{x-ray of $Z(s)$}
\label{fragment}
\end{center}
\end{figure}

Following a tradition that started with Gauss \cite{Gauss}*{plate in p.~31}, we present the X-ray of the function $Z(s)$. This is a plot of the lines where the function $Z(s)$ is real (continuous) or purely imaginary (dotted).

The $y$-axis is also plotted. The double pole is clearly visible at $s=1$. There are some dotted lines that cut the real axis approximately at the even negative integers $-2$, $-4$, $-6$, \dots These crossing points are zeros of $Z(s)$, but it is known that  $Z(-2n)=(-1)^n \frac{8-E_{2n}}{2^{2n+3}}\ne 0$. With our program we may compute these zeros. They are certainly very near the even integers. In the following table they are the numbers $a_2$, $a_4$, $a_6$, \dots

\bigskip
\begin{center}
\begin{tabular}{|c|l||c|l|}
\hline
\multicolumn{4}{|c|}{\bfseries Real Zeros of $Z(s)$}\\
\hline\label{tableofzeros}
$a_1$ & -0.99131855134306435 & $a_2$ & -1.87934753430942316\\
$a_3$ & -3.00020218979105365 & $a_4$ & -3.99946823514428604\\
$a_5$ & -4.99999673336932148 & $a_6$ & -5.99994478082934430\\
$a_7$ & -7.00000004296610568 & $a_8$ & -8.00000245703478924\\
$a_9$ & -8.99999999946336456 & $a_{10}$ & -9.99999987797814206\\
\hline
\end{tabular}
\end{center}

In this table there are other real zeros of $Z(s)$: $a_1$, $a_3$, $a_5$, \dots  These zeros are very near the poles at $-1$,  $-3$, $-5$, \dots  The reason nothing is visible in the plot at these points is that in fact there are closed lines that join the zeros and the poles. However they are so small that they fail to appear in the plot on  this scale.

Indeed there are two  cases. At the points $-1$, $-3$, $-7$, $-11$,  the X-ray contains only a closed (almost a circle) dotted line joining the zero and pole. At the points $-5$, $-9$, there are also two additional points $c_5$, $d_5$ and $c_9$, $d_9$ at which $Z'(s)$ vanishes.  The X-ray in these cases  also has an additional closed line (at which the function $Z(s)$ is real) which joins the zeros of the derivative (see Figure  \ref{sketch}), that is not drawn to scale, since again the two loops are very different in magnitude.

\begin{figure}[H]
  \includegraphics[width=6cm]{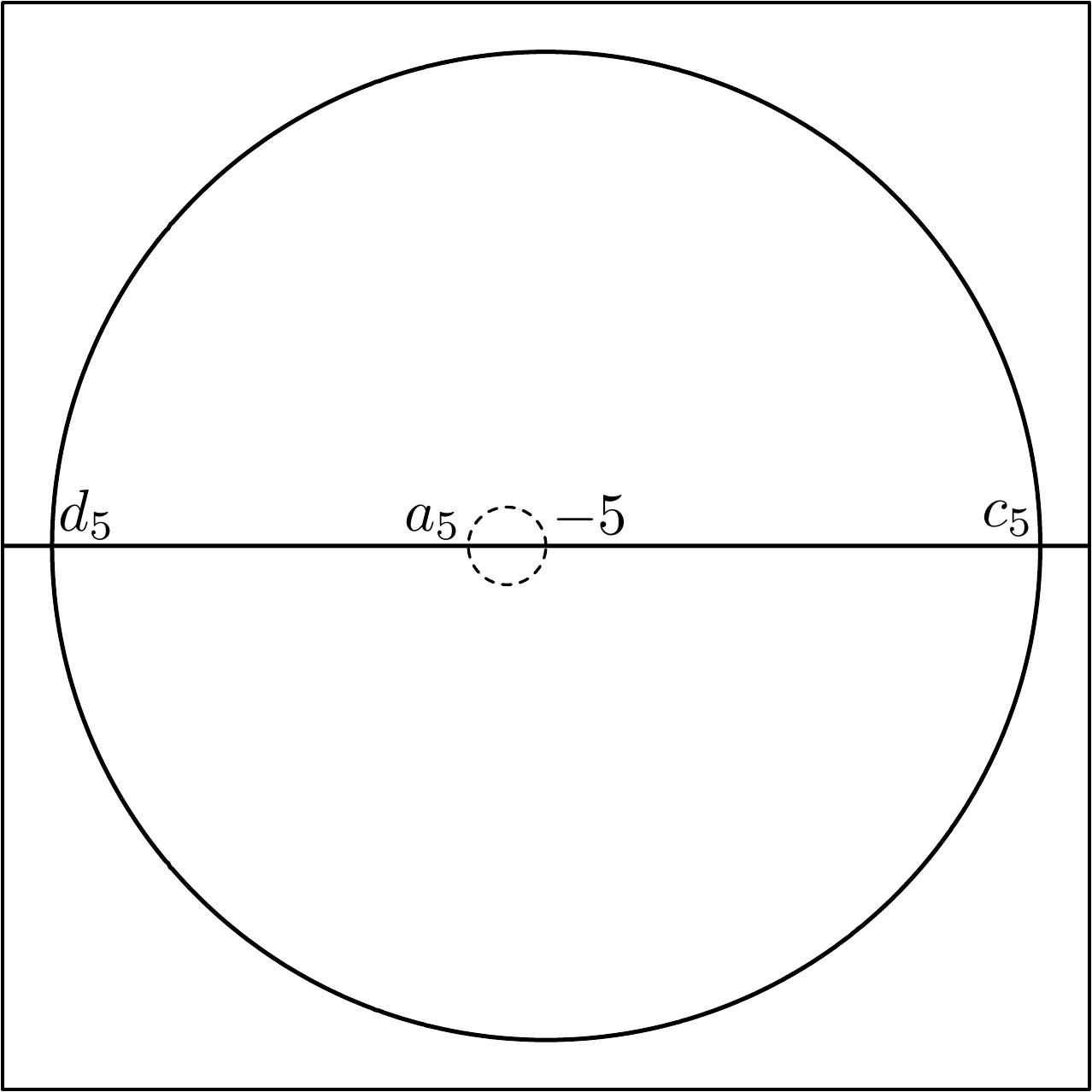}
  \includegraphics[width=6cm]{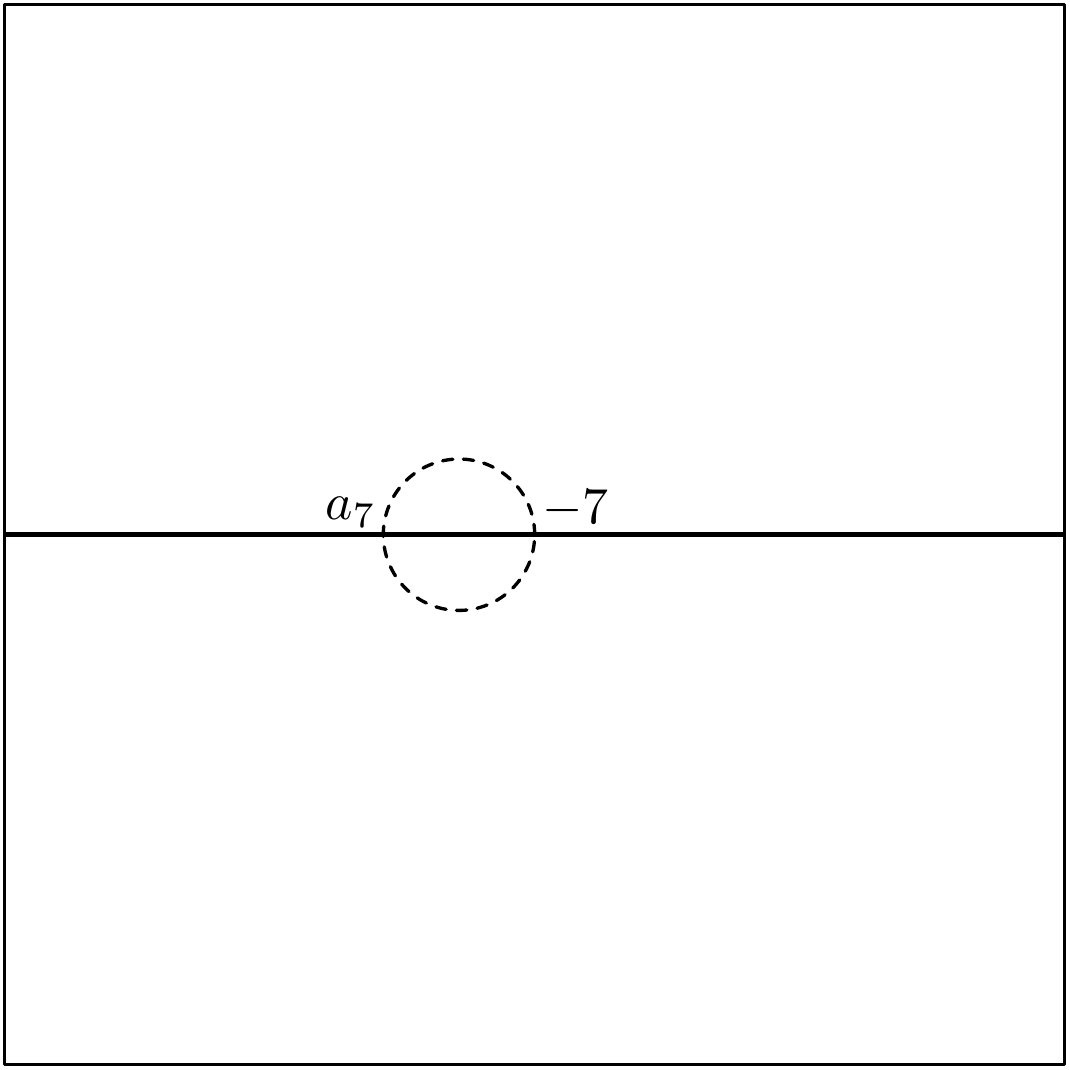}\\
  \caption{Sketch of X-ray  near $-5$ and near $-7$ (not at scale). }
  \label{sketch}
\end{figure}

Assuming the Riemann hypothesis all $\alpha_n$ are real $\alpha_n=\gamma_n$. In this case is easy to see that $Z(2+it)\approx \frac{1}{\gamma_1^2}e^{-it\log\gamma_1}$ where $\gamma_1=14.134725\dots$, the ordinate of the first non-trivial zero of $\zeta(s)$. Therefore in this case the X-ray contains parallel (real and imaginary) lines  on the right of the complex plane. On these lines the function $Z(s)$ is monotonous, and since $\lim_{\sigma\to+\infty}Z(\sigma+it)=0$  the function $Z(s)$ does not vanish on these lines. It follows that assuming the Riemann hypothesis if there are some non-real zeros of $Z(s)$ they must be situated on other lines, that do not pass through the half-plane $\sigma\ge2$. The X-ray suggest that these zeros do not exist. If the Riemann hypothesis is not true and, for example, $\alpha_n$ is not real, then $\alpha_n^{2+it}$ is not bounded for $t\to+\infty$. In this case it it is not clear to me that $Z(s)$ do not have zeros  $\rho$ with $\Re(\rho)>2$.

Voros in \cite{V2} sought plots of the function $Z(x)$ and provides only a somewhat poor plot of a related function on p.~125 based on a few computed values.  We now give  plots of a more complete nature. For a plot of $Z(x)$ for real $x$,  some useful facts are known:

\begin{figure}[H]
\begin{center}
\includegraphics[width=\hsize]{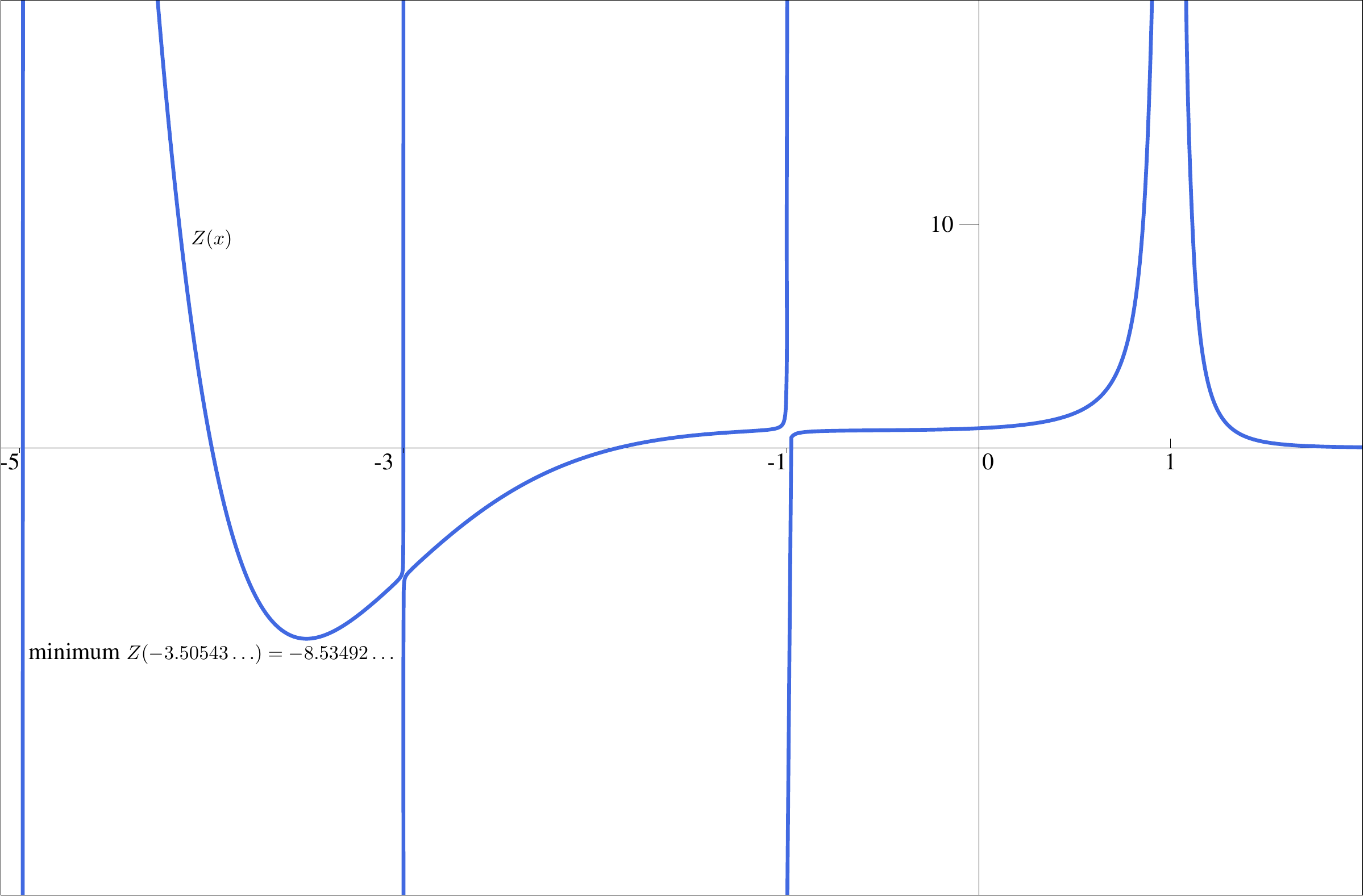}
\caption{Plot of  $Z(s)$ in the interval $(-5.1,2)$}
\label{firstplot}
\end{center}
\end{figure}

$Z$ has poles at each of the points $x=1-2n$ for $n=0$, $1$, $2$, \dots\ Therefore the graph of $Z(x)$ has asymptotes at $x=1-2n$.

Since the pole at $s=1$ is double and the main part is known to be 
$\frac{1}{2\pi(s-1)^2}-\frac{\log2\pi}{2\pi(s-1)}$,  it follows that 
$Z(1^+)=Z(1^-)=+\infty$. The rest of the poles are simple with main part 
$(-1)^n\frac{(1-2^{1-2n})}{2\pi}\frac{B_{2n}}{2n}\frac{1}{s+2n-1}$. 
All the residues are  negative. Therefore $Z(x)$ is increasing  near every pole $-1$, $-3$, $-5$, \dots

In Figure \ref{firstplot},   it can be  observed that at the poles $-3$ and $-1$,   the graph of the  function  appears similar to that of a regular function plus a straight vertical line. There is an extreme of the function at  point $-3.50543\dots$\ Finally, on the interval $(-5, -4.3)$,  the function is so large  that it disappear off the graph.  This is general, and hence the function must be plotted in small  segments,  see Figures \ref{secondplot} and \ref{secondplot2} and \ref{secondplot3}.

\begin{figure}[H]
\begin{center}
\includegraphics[width=0.92\hsize]{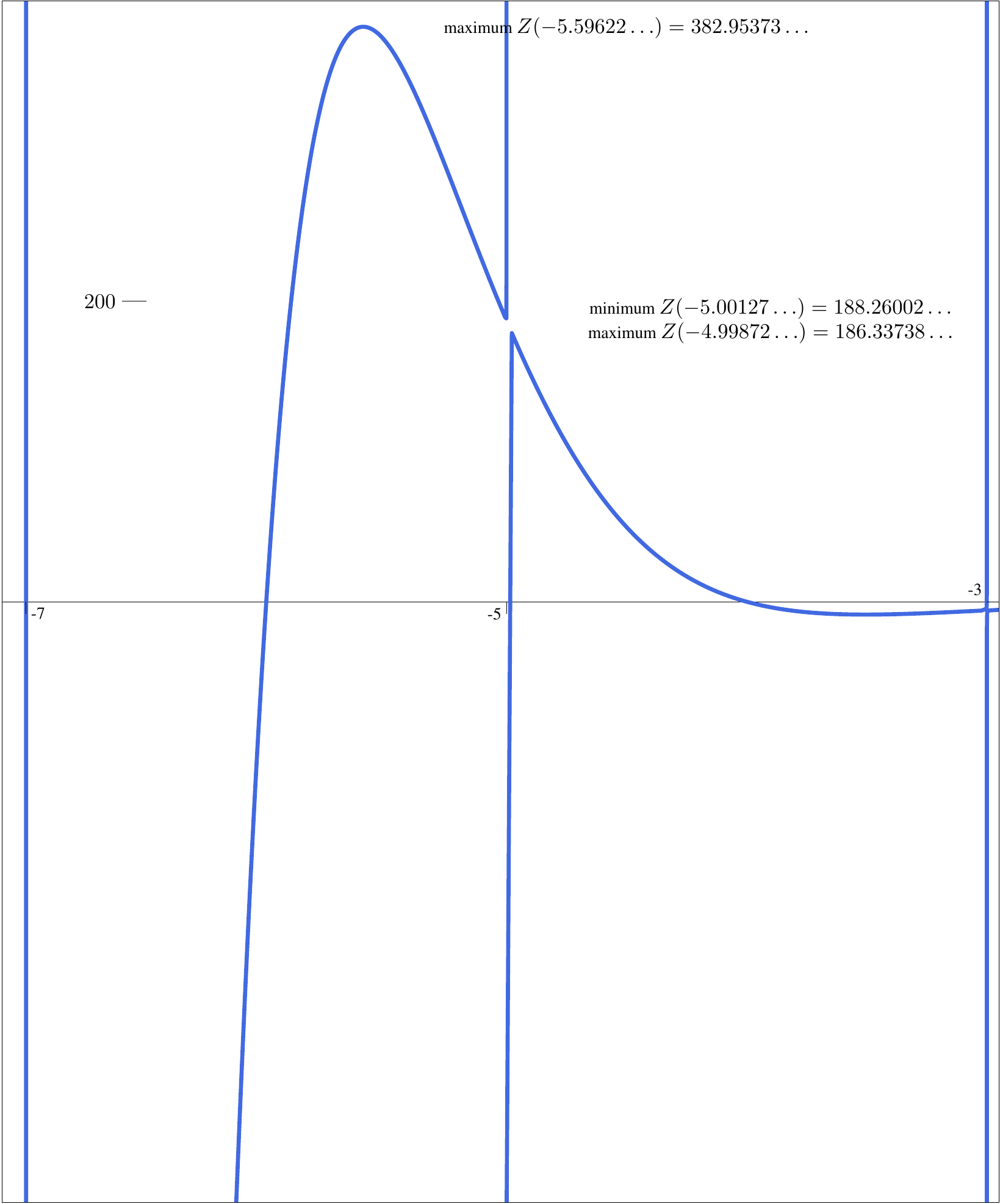}
\caption{Plot of  $Z(s)$ in the interval $(-7,-3)$}
\label{secondplot}
\end{center}
\end{figure}

The extremes of the function appear to be of two types:   A sequence $(b_{2n+1})_{n\ge1}$ that is  alternately  minimum and maximum 

\begin{figure}[H]
\begin{center}
\includegraphics[width=0.7\hsize]{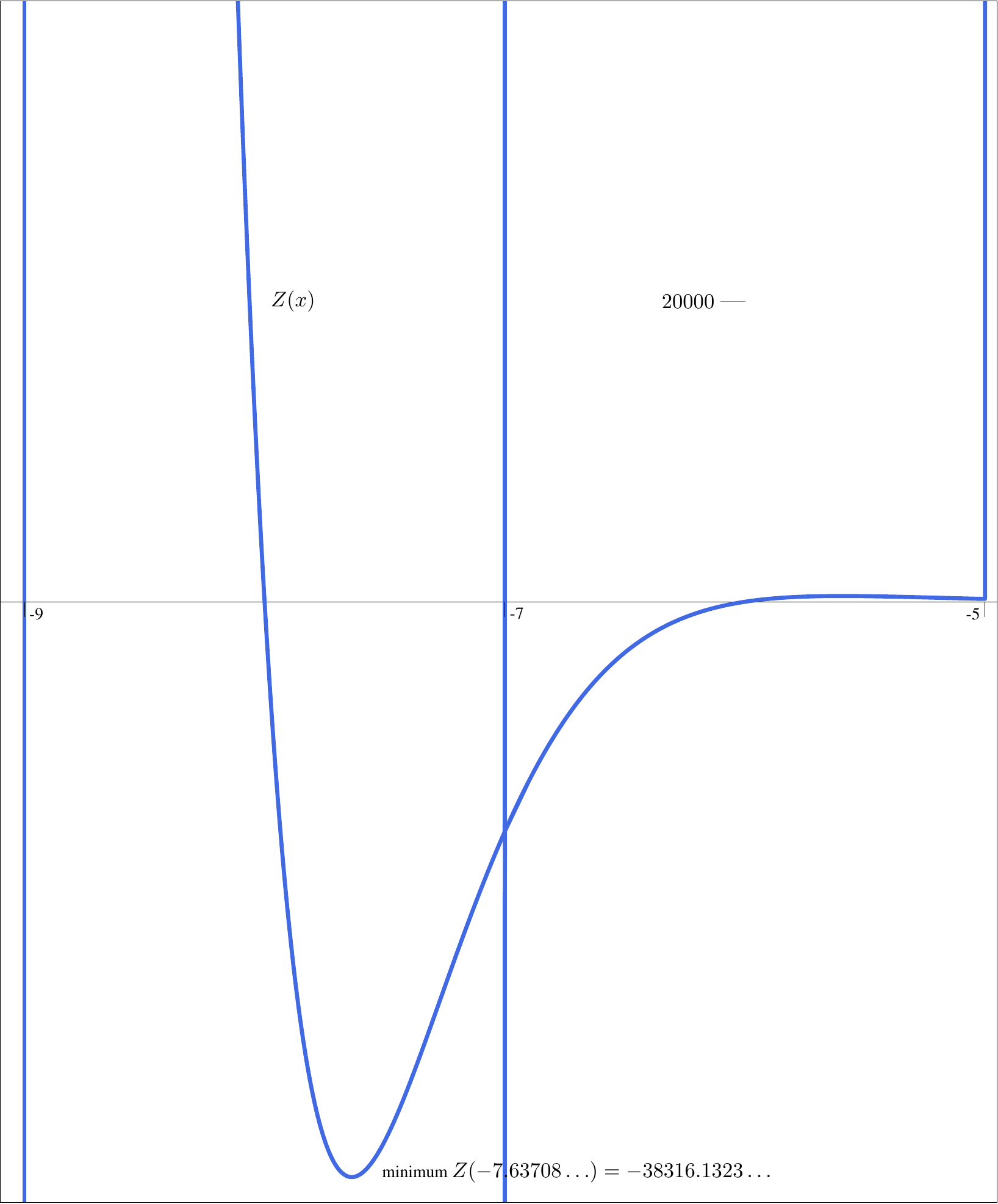}
\caption{Plot of  $Z(s)$ in the interval $(-9,-5)$}
\label{secondplot2}
\end{center}
\end{figure}

\begin{align*}
Z(x) \text{ has a minimum at } &b_3=-3.5054329756 \text{ with } Z(b_3)=-8.5349210063\\
Z(x) \text{ has a maximum at } &b_5=-5.5962251804 \text{ with } Z(b_5)=382.9537364130\\
Z(x) \text{ has a minimum at } &b_7=-7.6370844920 \text{ with } Z(b_7)=-38316.132343\\
Z(x) \text{ has a maximum at } &b_9=-9.6625777781 \text{ with } Z(b_9)=6547191.834042\\
\end{align*}

\noindent The second type of extreme consists of two sequences $c_{4n+1}$ and $d_{n+1}$, 
where $d_{4n+1}<-4n-1<c_{4n+1}$ whereby $d_{4n+1}$ is
a minimum and $c_{4n+1}$ a maximum. 

\begin{align*}
Z(x) \text{ has a maximum at } &c_5=-4.9987263742 \text{ with } Z(c_5)=186.337380852\\
Z(x) \text{ has a minimum at } &d_5=-5.0012721685 \text{ with } Z(d_5)=188.260023541\\
Z(x) \text{ has a maximum at } &c_9=-8.9999856651 \text{ with } Z(c_5)=2242256.38095\\
Z(x) \text{ has a minimum at } &d_9=-9.0000143345 \text{ with } Z(d_5)=2242592.17120
\end{align*}

\begin{figure}[H]
\begin{center}
\includegraphics[width=13cm]{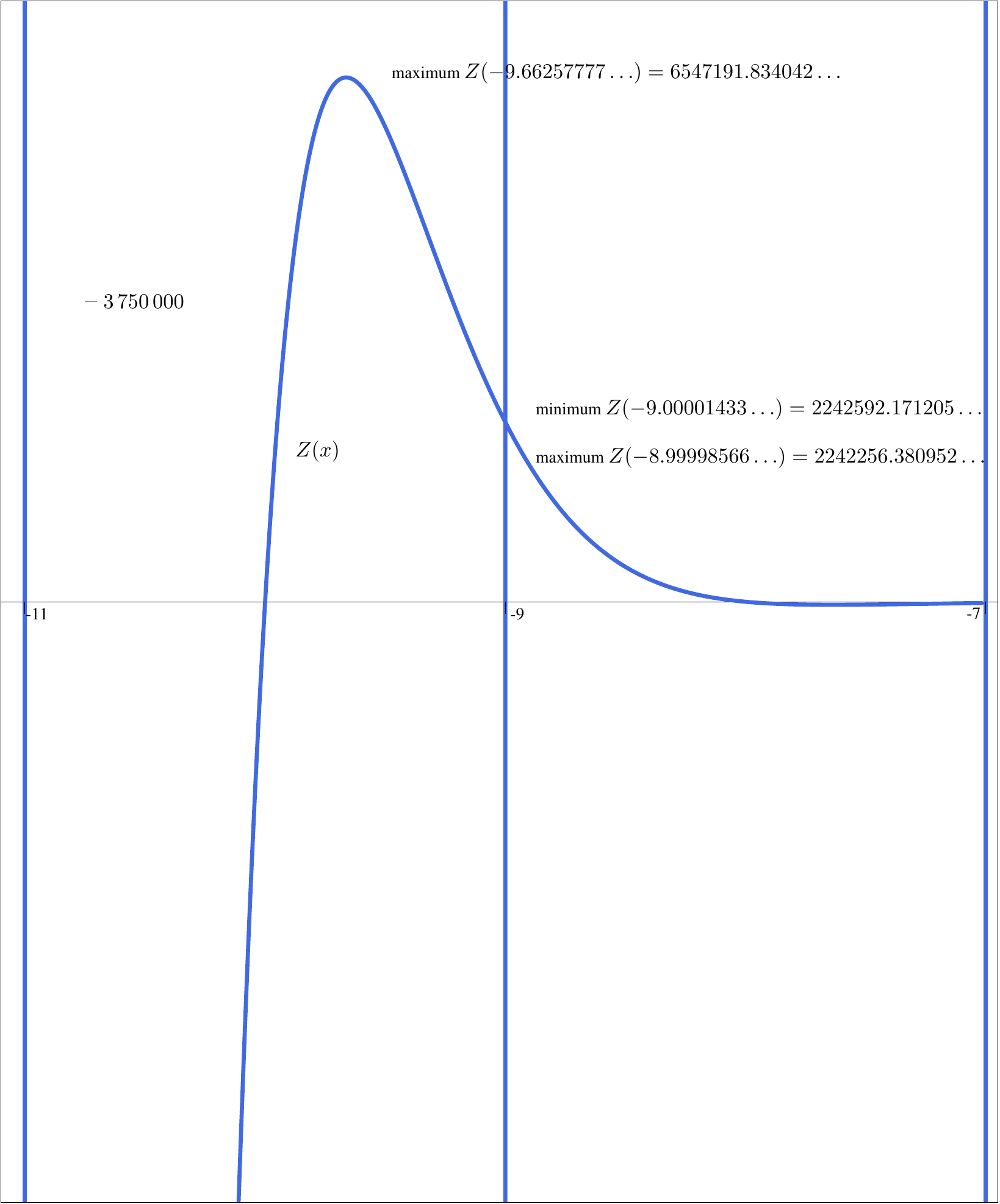}
\caption{Plot of  $Z(s)$ in the interval $(-11,-7)$}
\label{secondplot3}
\end{center}
\end{figure}

With our program and the capabilities of \text{mpmath} to compute to arbitrary precision,  these values may be computed with arbitrary precision. For example

\begin{align*}
c_5&=-4.998726374199056073034302436758537803935149729503553078275\\
d_5&=-5.001272168535916515656748009012378452508585888323143641629\\
c_9&=-8.999985665110044168150122555612459563666918804653018290127\\
d_9&=-9.000014334538926435239900710801932564852235975121014884881
\end{align*}

\begin{figure}[H]
\begin{center}
\includegraphics[width=0.9\hsize]{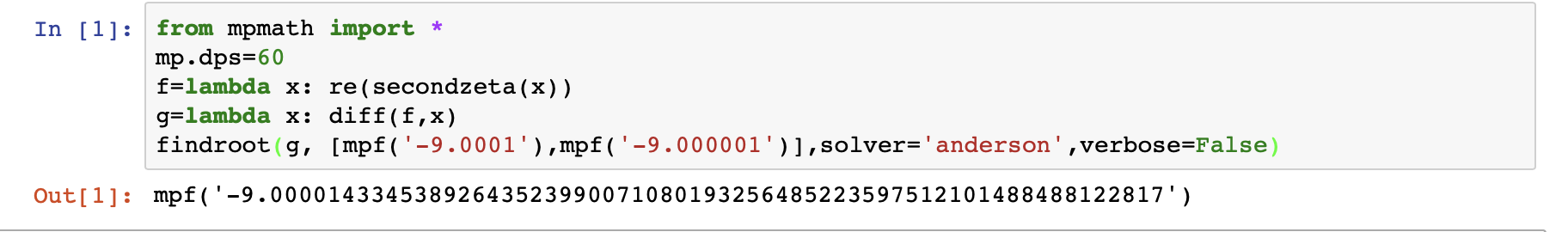}
\caption{Computing $d_9$ in Sage}
\label{computd9}
\end{center}
\end{figure}

Finally,  the apparent values of the function at the poles must coincide with the mean value of the function, or, equivalently,  the finite part of the function. These are easily computed by means of

\begin{figure}[H]
\begin{center}
\includegraphics[width=0.9\hsize]{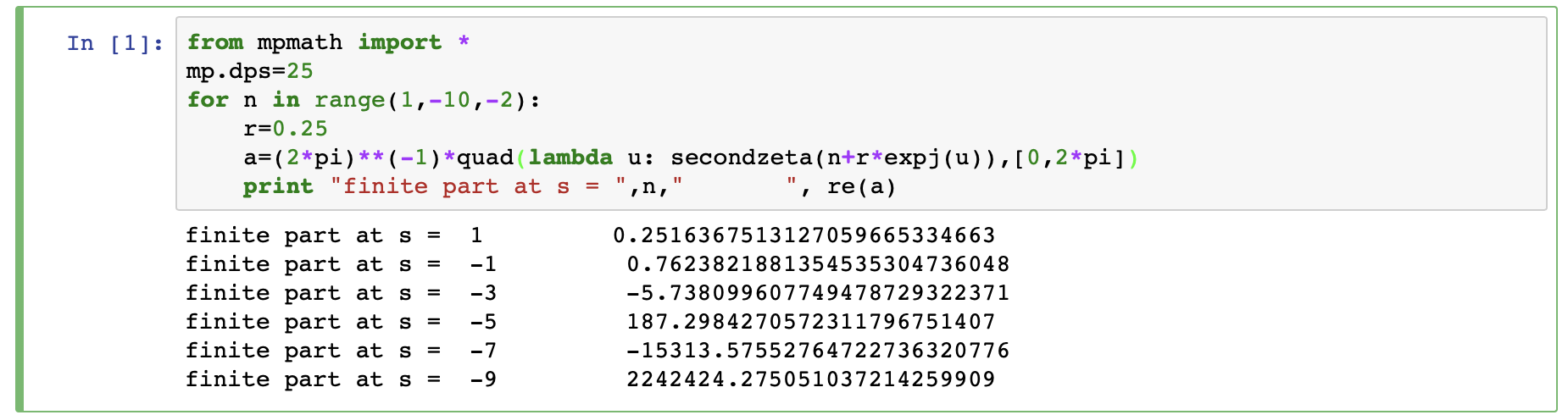}
\caption{Computation of the apparent value at $s=1-2n$}
\label{apparent}
\end{center}
\end{figure}
\noindent where the implied precision is expected to hold, but is not guaranteed.

The first  increasing and decreasing intervals of the function appear to be 
\begin{equation*}
d_9\nearrow-9\nearrow c_9\searrow b_7\nearrow-7\nearrow
b_5\searrow d_5\nearrow-5\nearrow c_5\searrow b_3\nearrow-1\nearrow 1\searrow+\infty
\end{equation*}


After seen these plots Voros  asked me to represent the function  $2Z(s)\cos\frac{\pi s}{2}$ that would be pole-free. We give here two pieces of this plot. This function vanishes at the points $-2n$ and at points $a_{2n+1}\approx-(2n+1)$

\begin{figure}[H]
\begin{center}
\includegraphics[width=0.9\hsize]{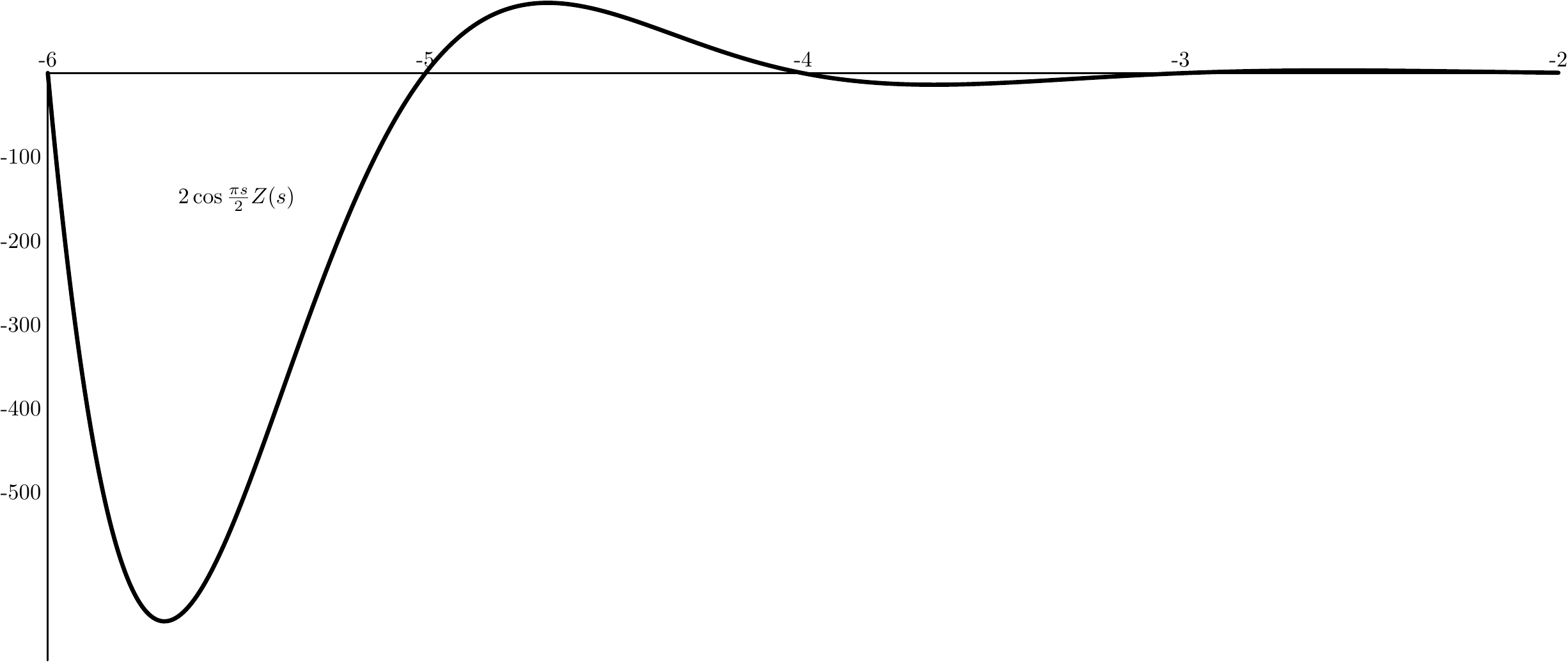}
\label{vorosplot}
\end{center}
\end{figure}
and have a pole at $s=1$
\begin{figure}[H]
\begin{center}
\includegraphics[width=0.9\hsize]{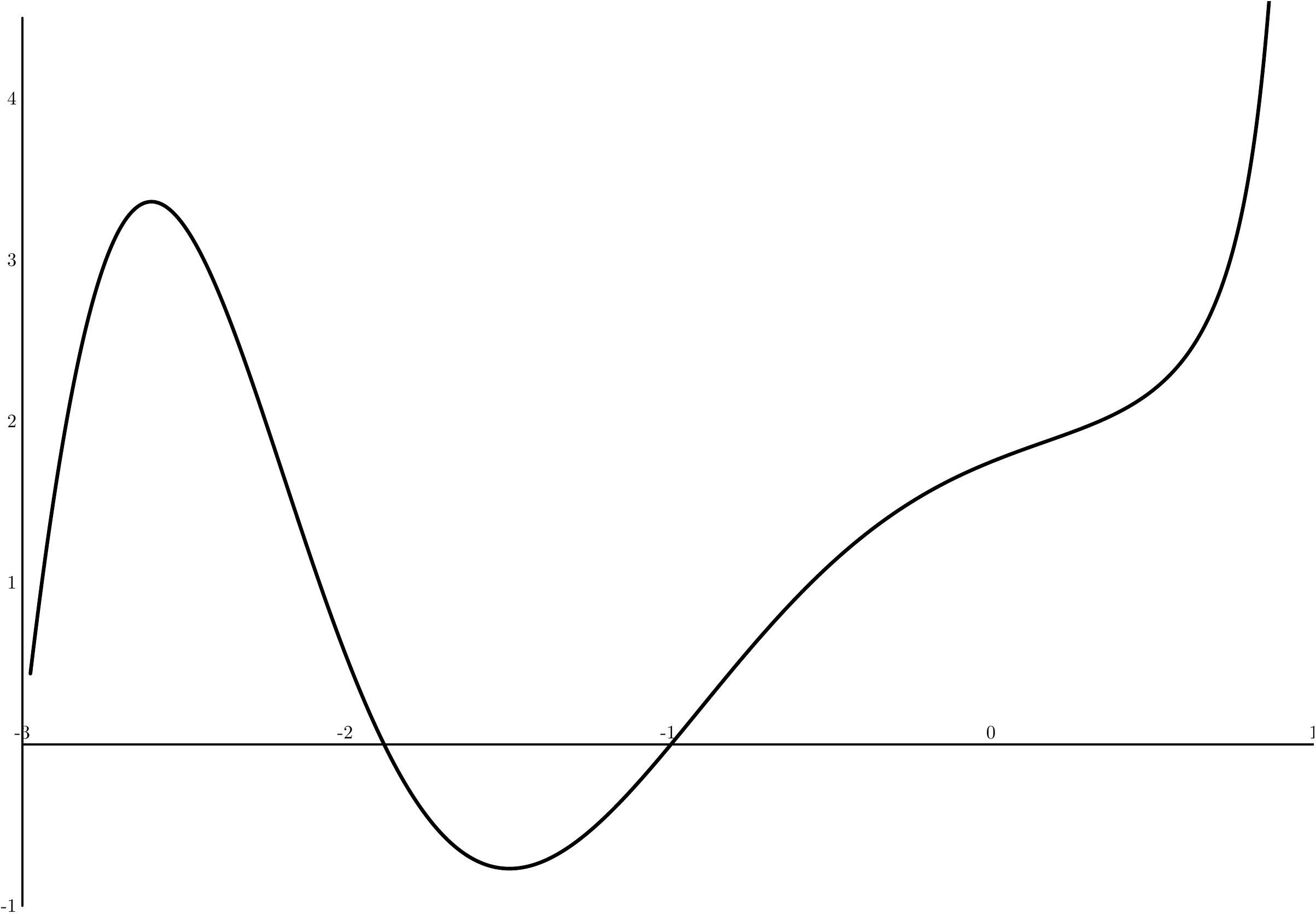}
\label{vorosplot2}
\end{center}
\end{figure}

\end{document}